\documentclass[12pt,leqno]{amsart}
\newtheorem{thm}{Theorem}[section]
\newtheorem{defi}{Definition}[section]

\newcommand{\be}{\begin{equation}}
\newcommand{\ee}{\end{equation}}
\numberwithin{equation}{section}
\newcommand{\bea}{\begin{eqnarray}}
\newcommand{\eea}{\end{eqnarray}}
\newcommand{\beb}{\begin{eqnarray*}}
\newcommand{\eeb}{\end{eqnarray*}}
\usepackage{amssymb,amsfonts,amsthm,setspace,indentfirst}
\usepackage[dvips]{graphics}
\usepackage{epsfig}

\begin{document}
\title{Rough convergence of sequences in a cone metric space}
\author{Amar Kumar Banerjee$^{1}$ and Rahul Mondal$^{2}$}
\address{\noindent\newline Department of Mathematics,\newline The University of
Burdwan, \newline Golapbag, Burdwan-713104,\newline West Bengal, India.}
\email{akbanerjee@math.buruniv.ac.in, akbanerjee1971@gmail.com}
\email{imondalrahul@gmail.com}

\begin{abstract}
Here we have introduced the idea of rough convergence of sequences in a cone metric space. Also it has been investigated how far several basic properties of rough convergence as valid in a normed linear space are affected in a cone metric space.
\end{abstract}
\noindent\footnotetext{$\mathbf{2010}$\hspace{5pt}AMS\; Subject\; Classification: 40A05, 40A99.}
\maketitle

\vspace{0.5in}

\section{\bf{Introduction}}
The idea of cone metric spaces was introduced by Huang and Zhang \cite{HLG} as a generalization of metric spaces. The distance $d(x,y)$ between two elements $x$ and $y$ in a cone metric space $X$ is defined to be a vector in a ordered Banach space $E$, where the order relation in the Banach space $E$ is given by using the idea of cone.\\
\indent In 2001 the notion of rough convergence of sequences was introduced  in a normed linear space by Phu \cite{PHU}. Phu discussed about the implication of rough convergence and the relation between ordinary convergence and rough convergence. There he also introduced the notion of rough Cauchy sequences. In 2003 Phu \cite{PHU1} discussed rough convergence of sequences in an infinite dimensional normed linear space as an extension of \cite{PHU}. The idea of rough statistical convergence was given by Ayter \cite{AYTER1} in 2008. Several works has been done in different direction \cite{AYTER2, PMROUGH1, PMROUGH2} by many authors using the idea given by Phu \cite{PHU}.\\
\indent In our work our motivation is to discuss the idea rough convergence of sequences in a cone metric space. But the structure of ordering associated of a cone metric space is the main problem to discuss this particular idea. However we have found out some basic properties of sequences regarding rough convergence in a cone metric space. 
\section{\bf{Preliminaries}}\label{preli}
\begin{defi}\cite{PHU}
Let $\{ x_{n} \}$ be a sequence in a normed linear space $(X, \left\| . \right\|)$,  and $r$ be a nonnegative real number. Then $\{ x_{n} \}$ is said to be $r$-convergent to $x$ if for any $\epsilon >0$, there exists a natural number $k$ such that $\left\|x_{n} - x \right\| < r +\epsilon$ for all $n \geq k$.
\end{defi}
\begin{defi} \cite{HLG}
Let $E$ be a real Banach space and $P$ be a subset of $E$. Then $P$ is called a cone if and only if $(i)$ $P$ is closed nonempty, and $P \neq \{0\}$\\
$(ii)$ $a,b \in \mathbb{R}$, $a,b \geq 0$, $x,y \in P$ implies $ax+by \in P$.\\
$(iii)$ $x \in P$ and $-x \in P$ implies $x=0$.
\end{defi}
Let $E$ be a real Banach space and $P$ be a cone in $E$. Let us use the partial ordering \cite{HLG} with respect to P by $x \leq y$ if and only if $y-x \in P$. We shall write $x <y$ to indicate that $x \leq y$ but $x \neq y$.\\
\indent Also by $x<< y$, we mean $y -x \in intP$, the interior of P. The cone P is called normal if there is a number $K >0$ such that for all $x,y \in E$, $0 \leq x \leq y$ implies $||x|| \leq K ||y||$. 
\begin{defi} \cite{HLG}
Let $X$ be a non empty set. If the mapping $d:X\times X \longrightarrow E$ satisfies the following three conditions\\ 
$(d1)$ $0 \leq d(x,y)$ for all $x,y \in X$ and $d(x,y)=0$ if and only if $x=y$;\\
$(d2)$ $d(x,y)=d(y,x)$ for all $x,y \in X$;\\
$(d3)$ $d(x,y) \leq d(x,z) + d(z,y)$ for all $x, y, z \in X$; then $d$ is called a cone metric on $X$, and $(X,d)$ is called a cone metric space.\\
\indent It is clear that a cone metric space is a generalization of metric spaces. Throughout $(X,d)$ or simply $X$ stands for a cone metric space which is associated with a real Banach space $E$ with a cone $P$, $\mathbb{R}$ for the set of all real numbers, $\mathbb{N}$ for the set of all natural numbers, Sets are always subsets of $X$ unless otherwise stated.\end{defi}
\begin{defi} \cite{HLG}
Let $(X, d)$ be a cone metric space. A sequence $\{ x_{n} \}$ in $X$ is said to be convergent to $x\in X$ if for every $c \in E$ with $0<<c$ there is $k\in \mathbb{N}$ such that $d(x_{n}, x)<<c$, whenever for all $n>k$.
\end{defi}
We now prove the following results which will be needed in the sequel.
\begin{thm}
Let $E$ be a real Banach space with cone $P$. If $x_{0} \in intP$ and $c(> 0)\in \mathbb{R}$ then $cx_{0} \in intP$.
\end{thm}
\begin{proof}
Let $x_{0} \in intP$. So there exists an open set $U$ such that $x_{0} \in U \subset P$. Since by definition $x,y \in P$ and $a,b \geq 0$ implies $ax+by \in P$, it follows that $cU \subset P$ for any real $c\geq 0$. Therefore $cx_{0} \in cU \subset P$. Now $cU$ is open and hence $cx_{0} \in intP$.
\end{proof}
\begin{thm}
Let $E$ be a real Banach space and $P$ be a cone in $E$. If $x_{0} \in P$ and $y_{0} \in intP$ then $x_{0} + y_{0} \in intP$.
\end{thm}
\begin{proof}
Let us consider the mapping $f:E \longrightarrow E$ be defined by $f(x)= x+ x_{0} $ for all $x \in E$. Clearly $f$ being a translation operator is a homeomorphism. Now since $y_{0} \in intP$, there exists an open set $U$ such that $y_{0} \in U \subset P$. Also $f(U)$ is open and hence $x_{0} + U$ is open. Also by definition $x_{0} + U \subset P$. Therefore $x_{0} + y_{0} \in x_{0} + U \subset P$. So $x_{0} + y_{0} \in intP$. 
\end{proof}
\noindent \textbf{Corollary 2.1.} If $x_{0}, y_{0} \in intP$ then $x_{0} + y_{0} \in intP$.\\
The following theorem is widely known.
\begin{thm}
A real normed linear space is always connected.
\end{thm}
With the help of above result we can deduce the following theorem.
\begin{thm} 
Let $E$ be a real Banach space with cone $P$. Then $0 \notin intP$.
\end{thm}
\begin{proof}
If possible let $0 \in intP$. Then for any $x \in P$ we have $0+x \in intP$ and hence $x \in intP$. Hence every element of $P$ is an interior point $P$ and consequently $P$ becomes an open set. Now $P$ is a non empty set which is both open and closed, This contradicts to the fact that E is connected.
\end{proof}
\section{\bf{Rough convergence in a cone metric space}}

\noindent \textbf{Definition 3.1.}\cite{PHU} Let $(X,d)$ be a cone metric space. A sequence $\left\{x_{n} \right\}$ in $X$ is said to be $r$-convergent to $x$ for some $r \in E$ with $0<<r$ or $r=0$ if for every $\epsilon$ with $(0<<)\epsilon$ there exists a $k \in \mathbb{N}$ such that $d(x_{n}, x)<<r + \epsilon$ for all $n \geq k$.\\
\indent Usually we denote it by $x_{n} \stackrel{r}{\rightarrow} x$, $r$ is said to be the roughness degree of rough convergence of $\left\{x_{n} \right\}$. It should be noted that when $r=0$ the rough convergence becomes the classical convergence of sequences in a cone metric space. If $\left\{x_{n} \right\}$ is $r$-convergent to $x$, then $x$ is said to be a $r$-limit point of $\left\{x_{n} \right\}$. From the next example we can observe that $r$-limit point of a sequence $\left\{x_{n} \right\}$ may not be unique. For some $r$ as defined above, the set of all $r$-limit points of a sequence $\left\{x_{n} \right\}$ is said to be the $r$-limit set of the sequence $\left\{x_{n} \right\}$ and we will denote it by $LIM^{r}x_{n}$. Therefore we can say $LIM^{r}x_{n}= \left\{x_{0} \in X : x_{n} \stackrel{r}{\rightarrow} x_{0}\right\}$.\\

\textbf{Example 3.1.} Let $X= \mathbb{R} ^{2}$ and $E= \mathbb{R} ^{2}$ with $ P =\{ (x,y)\in E : x,y \geq 0 \}$. Now let us define $d: X\times X \longrightarrow E $ by $d(x,y)= ( \left\| x-y \right\|, \left\| x-y \right\| )$ where $\left\| x-y \right\|= \sqrt{(x_{1} -y_{1})^{2} + (x_{2} -y_{2})^{2}} $ for $x= (x_{1}, y_{1}) \in \mathbb{R} ^{2}$ and $y= (x_{2}, y_{2}) \in \mathbb{R} ^{2}$. Clearly $(X, d) $ is a cone metric space with $P$ as cone and which is not a metric space.\\
\indent Now let us consider a sequence $\{ x_{n} \}$ in $X$ where $ x_{n} = (1,1)$ if $n$ is odd and $ x_{n} = (2,2)$ if $n$ is even. We will show that $\{ x_{n} \}$ is not convergent in the usual sense of cone metric spaces. Clearly $\{ x_{n} \}$ can not converge to $(1,1)$ because choosing $\epsilon =( \frac {1} {2}, \frac {1} {2})$ we have infinitely many natural numbers that is the even natural numbers for which $d(x_{n}, x) << \epsilon$ does not hold, where $x=(1,1)$. Similarly we can show that  $\{ x_{n} \}$ can not converge to $(2,2)$. Now if we suppose that $\{ x_{n} \}$ converges to $p= (x_{0} , y_{0})$ in $X$, where $p\neq (1,1)$ and $p\neq (2,2)$. Let $\sqrt {(x_{0}- 1)^{2} + (y_{0}- 1)^{2}}= c _{1}$ and $\sqrt {(x_{0}-2 )^{2} + (y_{0}-2)^{2}}= c _{2}$. We denote the minimum of $c _{1}$ and $c _{2}$ by $c$ and let $\epsilon = (c,c)$. Then clearly $0 << \epsilon$, but there exists infinitely many terms of the sequence for which the relation $d(x_{n}, p)<< \epsilon$ does not hold.\\
\indent If we consider $r= (\sqrt{2}, \sqrt{2})$ then $\{ x_{n} \}$ is $r$-convergent to $s=(1,1)$. Because $d(x_{n}, s)$ equals to $(\sqrt{2}, \sqrt{2})$ if $n$ is even and $d(x_{n}, s)$ equals to $(0,0)$ if $n$ is odd. Hence $d(x_{n}, s) << r+ \epsilon$ for all $n \in \mathbb{N}$ and for every $(0<<)\epsilon \in E$.\\
\indent We now recall the definition of boundedness \cite{MNM} of a sequence in a metric space as follows.\\
A sequence $\left\{x_{n} \right\}$ in a metric space $(X,d)$ is said to be bounded if for any fixed $a \in X$, there exists $r>0$ such that $d(x_{n},a)< r$ for all $n \in \mathbb{N}$. So $d(x_{n},x_{m}) \leq d(x_{n},a) + d(x_{m},a) < 2r$ for all $n,m \in \mathbb{N}$. Using this idea we define boundedness of a sequence in a cone metric space as follows.\\
\noindent \textbf{Definition 3.2.}$($cf.\cite{MNM}$)$ A sequence $\left\{x_{n} \right\}$ in $(X,d)$ is said to be bounded if there exists a $(0<<)k \in E$ such that $d(x_{m}, x_{n}) << k$ for all $m,n \in \mathbb{N}$.\\
\indent As in the case of a normed linear space it is also true in a cone metric space that if a sequence is bounded then the $r$-limit set of this sequence is non empty for some $r$ as defined above. The following theorem has been given in evidence of that.
\begin{thm}
If a sequence $\left\{x_{n} \right\}$ in $(X,d)$ is bounded then $LIM^{r}x_{n} \neq \phi$ for some $(0<<)r \in E$.
\end{thm}
\begin{proof}
Since $\left\{x_{n} \right\}$ is bounded there exists a $(0<<)p$ such that $d(x_{n},x_{m} ) << p$ for all $m,n \in \mathbb{N}$. Hence $p- d(x_{n},x_{m}) \in intP$ for all $m,n \in \mathbb{N}$. Therefore by theorem $2.2$ for every $(0<<)\epsilon \in E$ we have $[p- d(x_{n},x_{m})]+ \epsilon \in intP$ for all $m, n\in \mathbb{N}$. Hence $d(x_{n},x_{m})<< p+ \epsilon$ for all $m, n\in \mathbb{N}$. So for any $k \in \mathbb{N}$, $d(x_{n},x_{k})<< p+ \epsilon$ for all $n\in \mathbb{N}$. Therefore  $\left\{x_{n} \right\}$ is $p$-convergent to $x_{k}$ for all $k \in \mathbb{N}$. So $LIM^{p}x_{n} \neq \phi$.
\end{proof}
It should be noted that in the above theorem we can use any $(0<<)r \in E$ with $p<< r$ to show that $LIM^{r}x_{n} \neq \phi$.\\
\noindent \textbf{Remark 3.1.} In the example 3.1 let us consider $s= (1,1)$ and $s^{'}= (2,2)$. Then for $r= (\frac {1} {2 \sqrt{2}}, \frac {1} {2 \sqrt{2}})$, $LIM^{r} x_{n} = \phi$. Because if we suppose $\{ x_{n} \} $ is $r$-convergent to $a$ then for every $(0<<) \epsilon \in E$ there exists a $k \in \mathbb{N}$ such that $d(x_{n}, a)<< r+ \epsilon$ for all $n \geq k$. Let $\epsilon = (\frac {1} {8 \sqrt{2}}, \frac {1} {8 \sqrt{2}})$, so $r+ \epsilon = (\frac {5} {8 \sqrt{2}}, \frac {5} {8 \sqrt{2}})$. Since $x_{n}= s$ if $n$ is odd and $x_{n}= s^{'}$ if $n$ is even and $d(x_{n}, a)<< r+\epsilon$ for all $n\geq k$, clearly $d(s,a)<< r+ \epsilon$ and $d(s^{'},a)<< r+ \epsilon$. Now $d(s, s^{'}) \leq d(s,a) + d(s^{'},a)$ and hence $d(s, s^{'})<< 2(r+ \epsilon)$ that is $d(s, s^{'})<< (\frac {5} {4 \sqrt{2}}, \frac {5} {4 \sqrt{2}})$, which is a contradiction as $d(s, s^{'})=(\sqrt{2}, \sqrt{2})$. Therefore $LIM^{r} x_{n} = \phi$.\\ 

\indent In our previous theorem we have discussed about the choice of $r$ for which a bounded sequence in a cone metric space is $r$-convergent.
\begin{thm}
Every $r$-convergent sequence in a cone metric space $(X, d)$ is bounded.
\end{thm}
\begin{proof}
Let $\left\{x_{n} \right\}$ be a $r$-convergent sequence in $(X, d)$ and $r$-convergent to $x$. So for every $(0<<)\epsilon \in E$, there exists a $k \in \mathbb{N}$ such that $d(x_{n}, x) << r+\epsilon$ for all $n \geq k \longrightarrow (*)$. Let $M=\sum _{n=1} ^{k} d(x_{n}, x)$. Clearly $0\leq M$ and using theorem 2.2 we have $0 << M+ \epsilon$. For $n<k$, we have $M- d(x_{n}, x) \in P$. Again since $0<< r+\epsilon$, we have $[M- d(x_{n}, x)]+ (r+ \epsilon) \in intP$, that is, $(M+r+ \epsilon) -d(x_{n}, x) \in intP$. Hence $d(x_{n}, x)<< (M+r+ \epsilon)$ for all $n<k \longrightarrow (i)$.\\
We have by $(*)$, $(r+ \epsilon) - d(x_{n}, x) \in intP$ for all $n \geq k$ and also $M \in P$. Hence $M+[(r+ \epsilon) - d(x_{n}, x) ] \in intP$ for all $n \geq k$. Therefore, $d(x_{n}, x) << M+(r+ \epsilon)$ for all $n \geq k \longrightarrow (ii)$. Hence from $(i)$ and $(ii)$ we can write $d(x_{n}, x) << M+(r+ \epsilon)$ for all $n \in \mathbb{N}$.\\ 
\indent Now for $i,j \in \mathbb{N}$, we have $d(x_{i}, x_{j}) \leq d(x_{i}, x) + d(x_{j}, x)$, that is, $[d(x_{i}, x) + d(x_{j}, x)] - d(x_{i}, x_{j}) \in P \longrightarrow (iii)$. Also $(M+r+ \epsilon) - d(x_{i}, x) \in intP$ and $(M+r+ \epsilon) - d(x_{j}, x) \in intP$ and hence, $2(M+r+ \epsilon) - [d(x_{i}, x) + d(x_{j}, x)] \in intP \longrightarrow (iv)$. From $(iii)$ and $(iv)$ we have $2(M+r+ \epsilon) - d(x_{i}, x_{j}) \in intP$ and so $d(x_{i}, x_{j}) << 2(M+r+ \epsilon)$. Hence $\left\{x_{n} \right\}$ is bounded.
\end{proof}
\begin{thm}
Let $\left\{x_{n} \right\}$ and $\left\{y_{n} \right\}$ be two sequences in a cone metric space $(X, d)$ and let $\left\{y_{n} \right\}$ converges to $y \in X$. If there exists a $(0<<)r \in E$ such that $d(x_{i}, y_{i}) \leq r$ for all $i \in \mathbb{N}$. Then $\left\{x_{n} \right\}$ $r$-converges to $y$.
\end{thm}

\begin{proof}
Let $(0<<) \epsilon $ be pre-assigned. Since $\left\{y_{n} \right\}$ converges to $y$ there exists a $k \in \mathbb{N}$ such that $d(y_{n},y) << \epsilon$ for all $n \geq k$ and hence $\epsilon - d(y_{n},y) \in intP$ for all $n \geq k$. Also for all $n \geq k$, we have $d(x_{n}, y_{n}) \leq r$, so $r -d(x_{n}, y_{n}) \in P$ for all $n \geq k$. Therefore by theorem 2.2 $(r+ \epsilon)-[d(x_{n}, y_{n})+d(y_{n},y)] \in intP$ for all $n \geq k$. Again since $d(x_{n}, y) \leq d(x_{n}, y_{n})+ d(y_{n},y)$ for all $n \in \mathbb{N}$, $[d(x_{n}, y_{n})+ d(y_{n},y)]- d(x_{n}, y) \in P$ for all $n \in \mathbb{N}$. So for all $n \geq k$ we have $[(r+ \epsilon)-(d(x_{n}, y_{n})+d(y_{n},y))] + [(d(x_{n}, y_{n})+ d(y_{n},y))- d(x_{n}, y)] = (r+ \epsilon) - d(x_{n}, y) \in intP$ for all $n \geq k$. Therefore $d(x_{n}, y) << r+ \epsilon$ for all $n \geq k$ and hence $\left\{x_{n} \right\}$ $r$-converges to $y$.
\end{proof}
It has been seen in \cite{PHU} that the diameter of a $r$-limit set of a sequence in a normed linear space is not greater then 2$r$. We find out similar kind of property of a sequence in a cone metric space regarding rough convergence as follows.  
\begin{thm}
Let $\left\{x_{n} \right\}$ be a sequence in $(X, d)$. Then there does not exists elements $y,$ $z$ in $LIM^{r}x_{n}$ such that $mr< d(x, z)$, where m is a real number greater then $2$.
\end{thm}

\begin{proof}
If possible let there exists elements $y, z \in LIM^{r}x_{n}$ such that $mr < d(y,z)$ and $m>2$. Let $(0<<)\epsilon$ be arbitrary.
Since $y, z \in LIM^{r}x_{n}$, there exists a $i \in \mathbb{N}$ such that $d(x_{i}, y)<< r + \frac {\epsilon}  {2}$, that is $(r +
\frac {\epsilon}  {2}) - d(x_{i}, y) \in intP \longrightarrow  (i)$ and  $d(x_{i}, z)<< r + \frac {\epsilon}  {2}$, that is $(r + \frac
{\epsilon}  {2}) - d(x_{i}, z) \in intP \longrightarrow  (ii)$. Hence from $(i)$ and $(ii)$ we can write $(2r + \epsilon)- [d(x_{i},
y) + d(x_{i}, z)] \in intP \longrightarrow (iii)$.\\
\indent Now $ d(y, z) \leq d(x_{i}, y) + d(x_{i}, z)$, hence $[d(x_{i}, y) + d(x_{i}, z)] - d(y, z) \in intP \longrightarrow (iv)$.
So from $(iii)$ and $(iv)$, $(2r + \epsilon) - d(y, z) \in intP \longrightarrow (v)$. Again $d(y, z) -mr \in P \longrightarrow
(vi)$. Hence from $(v)$ and $(vi)$ we can write $(2r + \epsilon) -mr \in intP$ that is $\epsilon - r(m-2) \in intP$. This is true for any
$0<< \epsilon$. So choosing $\epsilon = r(m-2)$, we have $0 \in intP$. This is a contradiction and the result follows.
\end{proof}
\begin{thm}
Let $\{ x_{n} \}$ be $r_{1}$-convergent to $x$ in $(X, d)$. Then $\{x_{n} \}$ is also $r_{2}$-convergent to $x$ in $(X, d)$ for $r_{1}<
r_{2}$.
\end{thm}

The proof is obvious and so is omitted.\\
\noindent \textbf{Corollary 3.1.} Let $\{ x_{n} \}$ be $r_{1}$-convergent to $x$ in $(X, d)$ and $r_{1} < r_{2}$ for some $0<<r_{2}$. Then $LIM^{r_{1}} x_{n} \subset LIM^{r_{2}} x_{n} $.\\
\noindent \textbf{Definition 3.3.} Let $\{ x_{n} \}$ be a sequence in a cone metric space $(X, d)$, a point $c \in X$ is said to be a cluster point of $\{ x_{n} \}$ if for every $(0<<)\epsilon \in E$ and for every $k \in \mathbb{N}$, there exists a $k_{1} \in \mathbb{N}$ such that $k_{1} > k$ with $d(x_{k_{1}}, c) << \epsilon$.\\
 \indent For $0<<r$ and a fixed $y \in X$, we define the sets $\overline {B_{r}(y)}$ and $B_{r}(y)$, the closed and open spheres respectively centred at $y$ with radius $r$ as follows:\\
$\overline {B_{r}(y)}= \{ x\in X : d(x,y) \leq r \}$ and $B_{r}(y) = \{ x\in X : d(x,y) << r \}$.
\begin{thm}
Let $(X,d)$ be a cone metric space, $c \in X$ and $0<< r$ be such that for any $x \in X$, either $d(x,c) \leq r$ or $r<< d(x,c)$. If $c$ is a cluster point of a sequence $\{x_{n} \}$ then $LIM^{r}x_{n} \subset \overline {B_{r}(c)}$.
\end{thm}
\begin{proof}
Let $y \in LIM^{r} x_{n}$ but $y \notin \overline {B_{r}(c)}$. So $r<< d(y,c)$. Let $\epsilon ^{'} = d(y, c)- r (\in intP)$ and so $d(y,
c)=r+\epsilon ^{'}$, where $0<< \epsilon ^{'}$. Let $\epsilon =\frac {\epsilon ^{'}}{2}$ and so we can write $d(y, c)= r+ 2 \epsilon$.
Then $B_{r+ \epsilon} (y) \cap B_{\epsilon} (c) = \phi$. Otherwise if $p \in B_{r+ \epsilon} (y) \cap B_{\epsilon}
(c)$ then $d(p,y)<< r+ \epsilon$ and $d(p,c)<< \epsilon$ and hence $(r+ \epsilon)- d(p,y)\in intP$ and $\epsilon - d(p,c) \in intP$. So we
have $(r+ 2 \epsilon)- [d(p,y)+d(p,c)] \in intP \longrightarrow (i)$. Again $d(y,c) \leq d(y,p)+ d(p,c)$ that is $[d(y,p)+ d(p,c)]-
d(y,c) \in P \longrightarrow (ii)$. Hence from $(i)$ and $(ii)$ we can write $(r+ 2 \epsilon) -d(y,c)=0 \in intP$, which is a contradiction as $0 \notin intP$. Therefore $B_{r+ \epsilon} (y) \cap B_{\epsilon} (c) = \phi$. But since $y \in LIM^{r}x_{n}$, for $0<< \epsilon$ there exists a $k_{0} \in \mathbb{N}$ such that $d(x_{n} , y) << r+ \epsilon$ for all $n \geq k_{0}$. Again since $c$ is a cluster point of $\{ x_{n} \}$, for $0<< \epsilon $ and for $k_{0} \in \mathbb{N}$, there exists a $k_{1}\in \mathbb{N}$, with $k_{1} > k_{0}$ such that $d(x_{k_{1}} , c)<< \epsilon$. So $x_{k_{1}} \in B_{\epsilon}(c)$. Also $d(x_{k_{1}}, y) << r + \epsilon$. So $x_{k_{1}} \in B_{r+ \epsilon} (y)$. Thus $x_{k_{1}} \in B_{r+ \epsilon} (y) \cap B_{\epsilon} (c)$. Which is a contradiction . Hence $y \in \overline{B_{r}(c)}$.
\end{proof}
\begin{thm}
Let $\left\{ x_{n} \right\}$ be a sequences in $(X,d)$ and $\left\{ y_{n} \right\}$ be a convergent sequence in $LIM^{r} x_{n}$ converging to $y_{0}$. Then $y_{0}$ must belongs to $LIM^{r} x_{n}$.
\end{thm}
\begin{proof}
Let $(0<<)\epsilon$ be preassigned. Since $\left\{ y_{n} \right\}$ converges to $y_{0}$, for $(0<<)\epsilon$ there exists a $k_{1} \in \mathbb{N}$ such that $d(y_{n}, y_{0}) << \frac { \epsilon}{2}$ for all $n \geq k_{1} \longrightarrow (i)$. Now let us choose $y_{m} \in LIM^{r} x_{n}$ with $m > k_{1}$. Then there exists a $k_{2} \in \mathbb{N}$ such that $d(x_{n}, y_{m}) << r+ \frac { \epsilon}{2}$ for all $n \geq k_{2} \longrightarrow (ii)$. Also for all $n \in \mathbb{N}$ we have, $d(x_{n}, y_{0}) \leq d(x_{n}, y_{m}) + d(y_{m}, y_{0})$. Hence we have $[d(x_{n}, y_{m}) + d(y_{m}, y_{0})] - d(x_{n}, y_{0}) \in P $ for all $n \in \mathbb{N} \longrightarrow (iii)$. Since $m > k_{1}$, by $(i)$ we can write $\frac { \epsilon}{2} -d(y_{m}, y_{0}) \in int P \longrightarrow (iv)$ and for all $n \geq k_{2}$ from $(ii)$ we have $(r+ \frac { \epsilon}{2}) - d(x_{n}, y_{m}) \in int P \longrightarrow (v)$.

Therefore for all $n \geq k_{2}$, by $(iv)$ and $(v)$ we can write $(r+ \epsilon) - [d(x_{n}, y_{m}+d(y_{m}, y_{0}) ] \in int P \longrightarrow (vi)$. Hence for all $n \geq k_{2}$, from $(iii)$ and $(vi)$ we have $(r+ \epsilon) - d(x_{n},y_{0}) \in intP $, that is $d(x_{n},y_{0})<<(r+ \epsilon)$ for all $n \geq k_{2}$. Hence $y_{0} \in LIM^{r} x_{n}$.
\end{proof}
\begin{thm}
Let $\left\{ x_{n} \right\}$ and $\left\{ y_{n} \right\}$ be two sequences in $(X,d)$. If for every $(0<<) \epsilon$ there exists a $k \in \mathbb{N}$ such that $d(x_{n}, y_{n}) \leq \epsilon$ for all $n \geq k$. Then $\left\{ x_{n} \right\}$ is $r$-convergent to $x$ if and only if $\left\{ y_{n} \right\}$ is $r$-convergent to $x$.
\end{thm}
\begin{proof}
Let $\left\{ x_{n} \right\}$ be $r$-convergent to $x$ and $(0<<) \epsilon$ be preassigned. Then for $0<< \epsilon$ there exist $k_{1}$, $k_{2} \in \mathbb{N}$ such that $d(x_{n}, x) << r+ \frac { \epsilon}{2}$ for all $n \geq k_{1}$ and $d(x_{n}, y_{n}) \leq \frac { \epsilon}{2}$ for all $n \geq k_{2}$. Let $k> max\left\{k_{1}, k_{2}\right\}$. Then for all $n \geq k$ we have $d(x_{n}, x) << r+ \frac { \epsilon}{2}$ and $d(x_{n}, y_{n}) \leq \frac { \epsilon}{2}$, that is we have 
\begin{center} $(r+ \frac { \epsilon}{2}) - d(x_{n}, x) \in int P$ for all $n \geq k     \longrightarrow (i)$ \end{center}
\begin{center}and  $[\frac { \epsilon}{2} - d(x_{n}, y_{n})] \in P$ for all $n \geq k     \longrightarrow (ii)$ \end{center}
Again for all $n \in \mathbb{N}$ we can write $d(y_{n}, x) \leq d(x_{n}, y_{n}) + d(x_{n}, x)$. Hence we have $[d(x_{n}, y_{n}) + d(x_{n}, x)] - d(y_{n}, x) \in P$ for all $n \in \mathbb{N}  \longrightarrow (iii)$. Therefore from $(i)$, $(ii)$ and $(iii)$ it follows that $(r + \epsilon) - d(y_{n}, x) \in intP$ for all $n \geq k$ and hence $d(y_{n}, x) << r + \epsilon$ for all $n \geq k$. Therefore $\left\{ y_{n} \right\}$ is $r$-convergent to $x$. Interchanging the roll of $\left\{ x_{n} \right\}$ and $\left\{ y_{n} \right\}$ it can be shown that if $\left\{ y_{n} \right\}$ is $r$-convergent to $x$ then $\left\{ x_{n} \right\}$ is $r$-convergent to $x$.
\end{proof}

\begin{thm}
Let $\left\{x_{n_{k}}\right\}$ be a sub sequence of $\left\{ x_{n} \right\}$ then $LIM^{r}  x_{n} \subset LIM^{r} x_{n_{k}}$.
\end{thm}
\begin{proof}

Let $y \in LIM^{r} x_{n}$ and $(0<<)\epsilon$ be arbitrary. Then there exists a $m \in \mathbb{N}$ such that $d(x_{n}, y)<< r+\epsilon$ for all $n \geq m$. Let $n_{p} > m$ for some $p \in \mathbb{N}$. Then $n_{k}>m$ for all $k \geq p$. Therefore $d(x_{n_{k}}, y)<< r+ \epsilon$ for all $k>p$. Hence $y \in LIM^{r} x_{n_{k}}$.
\end{proof}

\begin{thm}
Let $\mathcal{C}$ be the set of all cluster points of a sequence $\left\{ x_{n} \right\}$ in $(X, d)$. Also let $0<< r$ be such that for any $x \in X$ either $d(x, c) \leq r$ or $r<< d(x, c)$ for each $c \in \mathcal{C}$. Then $LIM^{r} x_{n} \subset \bigcap_{c \in \mathcal{C}} \overline{B_{r}(c)} \subset \left\{ x_{0} \in X : \mathcal{C} \subset \overline{B_{r}(x_{0})} \right\}$.
\end{thm}
\begin{proof}
By theorem 3.6 we can say that $LIM^{r} x_{n} \subset \bigcap_{c \in \mathcal{C}} \overline{B_{r}(c)}$. Now let $ y \in \bigcap_{c \in \mathcal{C}} \overline{B_{r}(c)}$. So $y \in \overline{B_{r}(c)}$ for each $c \in \mathcal{C}$ and hence $d(y,c) \leq r$ for each $c\in \mathcal{C}$. This implies that $c \in \overline{B_{r}(y)}$ for each $c \in \mathcal{C}$. Therefore $\mathcal{C} \subset \overline{B_{r}(y)}$. So $\bigcap_{c \in \mathcal{C}} \overline{B_{r}(c)} \subset \left\{ x_{0} \in X : \mathcal{C} \subset \overline{B_{r}(x_{0})} \right\}$. Therefore we have $LIM^{r} x_{n} \subset \bigcap_{c \in \mathcal{C}} \overline{B_{r}(c)} \subset \left\{ x_{0} \in X : \mathcal{C} \subset \overline{B_{r}(x_{0})} \right\}$.
\end{proof}
\noindent \textbf{Lemma 3.1.} Let $(X,d)$ be a cone metric space, where $P$ is a normal cone with normal constant $k$. Then for every
$\epsilon >0$, we can choose $c \in E$ with $c \in int P$ and $k\left\|c\right\| < \epsilon$.

\begin{proof}
Let $\epsilon >0$ be given and let $x \in int P$. We consider $c=  \frac {x.\frac{\epsilon}{ 2}}  {\left\|x\right\| .k}$. Clearly $c \in int P$ and also $\left\|c\right\|= \frac{\epsilon} {2.k} < \frac{\epsilon} {k}$. Hence $k\left\|c\right\| < \epsilon$.
\end{proof}
\noindent \textbf{Lemma 3.2.} Let $(X,d)$ be a cone metric space, $P$ be a normal cone with normal constant $k$. Then for each $c \in E$ with $0<<c$, there is a $\delta > 0$, such that $\left\|x\right\| < \delta $ implies $c- x \in int P$.

\begin{proof}

Since $0<<c$, $c \in int P$ and so there exists a $\delta >0$, such that $B_{\delta}(c)=\{x \in E: \left\|x-c\right\| < \delta \} \subset P$. Now $\left\| x \right\| < \delta$ implies $\left\|c-(c- x) \right\| < \delta$. Hence $c- x \in B_{\delta}(c)$ and therefore $c-x \in int P$.

\end{proof}

\begin{thm}
Let $(X,d)$ be a cone metric space, $P$ be a normal cone with normal constant $k$. Then a sequence $\left\{x_{n}\right\}$ in $(X,d)$ is $r$-convergent to $x$ if and only if $\left\{ d(x_{n},x)-r\right\}$ converges to $0$ in $E$.
\end{thm}

\begin{proof}
First suppose that $\left\{x_{n}\right\}$ is $r$-convergent to $x$. Let $\epsilon >0$ be preassigned. Then we have an element $c \in E$ with $0<<c$ and $k \left\| c \right\| < \epsilon$. Now for $0<<c$ there exists a $k_{1} \in \mathbb{N}$ such that $d(x_{n},x)<<r+c$ for all $n \geq k_{1}$. Hence $d(x_{n},x)-r<<c$ for all $n \geq k_{1}$. Now $\left\| d(x_{n},x)-r \right\| \leq k \left\| c \right\| < \epsilon$ for all $n \geq k_{1}$. Therefore $\left\{ d(x_{n},x)-r\right\}$ converges to $0$ in $E$.\\

Conversely let $\left\{ d(x_{n},x)-r\right\}$ converges to $0$ in $E$. For any $c \in E$ with $0<<c$, there is a $\delta >0 $, such that $\left\| x \right\| <\delta$ implies $c- x \in intP$. For this $\delta$ there is a $k_{1} \in \mathbb{N}$, such that $\left\| d(x_{n},x)-r\right\| < \delta$ for all $n \geq k_{1}$. So $c- [d(x_{n},x)-r] \in intP$ for all $n \geq k_{1}$. Hence  $d(x_{n},x)<<r+c$ for all $n \geq k_{1}$. Therefore $\left\{x_{n}\right\}$ in $(X,d)$ is $r$-convergent to $x$.
\end{proof}
\begin{thm}
Let $(X,d)$ be a cone metric space with normal cone $P$ and normal constant $k$. If $\{x_{n} \}$ and $\{y_{n} \}$ be two sequences $\frac{1} {4k+2}r$-convergent to $x$ and $y$ respectively in $X$, then the sequence $\{z_{n} \}$ is $\left\| r \right\|$-convergent to $d(x,y)$ where $0<<r$ and $z_{n}= d(x_{n}, y_{n})$ for all $n \in \mathbb{N}$.
\end{thm}
\begin{proof}
Let $\epsilon > 0$ be preassigned and let $x \in intP$. Then $c= \frac {x \frac{\epsilon}{2}} {\left\|x\right\|(4k+2)} \in intP$ and $\left\|c \right\| < \frac {\epsilon} {4k+2}$. Since $\{x_{n} \}$ and $\{y_{n} \}$ are $\frac{1} {4k+2}r$-convergent to $x$ and $y$ respectively, for $0<< c$, there exists a $k \in \mathbb{N}$ such that $d(x_{n}, x)<<c +\frac{r}{4k+2}$ and $d(y_{n}, y)<<c +\frac{r}{4k+2}$ for all $n \geq k$.
Therefore we have $(c +\frac{r}{4k+2})- d(x_{n}, x) \in intP$ and $(c +\frac{r}{4k+2})- d(y_{n}, y) \in intP$ for all $n \geq k$. Hence $2(c +\frac{r}{4k+2})-[d(x_{n}, x)+d(y_{n}, y)]\in intP$ for all $n \geq k$. Now denoting $\frac{r}{4k+2}$ by $r^{'}$, we have $2(c +r^{'})-[d(x_{n}, x)+d(y_{n}, y)]\in intP$ for all $n \geq k \longrightarrow (i)$.\\

Now $d(x,y) \leq d(x_{n}, x) + d(x_{n}, y)$ for all $n \in \mathbb{N}$, that is $[d(x_{n}, x) + d(x_{n}, y)]-d(x,y) \in P$ for all $n \in \mathbb{N} \longrightarrow (A)$. Also $d(x_{n}, y)\leq d(x_{n}, y_{n})+d(y_{n}, y)$ for all $n \in \mathbb{N}$. Hence $[d(x_{n}, y_{n})+d(y_{n}, y)]- d(x_{n}, y) \in P$ for all $n \in \mathbb{N} \longrightarrow(B)$. From $(A)$ and $(B)$ we have $[d(x_{n}, x) + d(y_{n}, y)+d(x_{n}, y_{n})]-d(x,y) \in P$ for all $n \in \mathbb{N} \longrightarrow (ii)$. Again $d(x_{n}, y_{n}) \leq d(x_{n}, x) + d(y_{n}, y)+d(x,y)$ for all $n \in \mathbb{N}$. So $[d(x_{n}, x) + d(y_{n}, y)+d(x,y)]-d(x_{n}, y_{n}) \in P$ for all $n \in \mathbb{N} \longrightarrow (iii)$. Now from $(i)$ and $(iii)$ we can write $2(c +r^{'})+[d(x,y)- d(x_{n}, y_{n})] \in intP$ for all $n \geq k \longrightarrow (iv)$. Now from $(i)$ and $(ii)$ we have $2(c +r^{'}) + [d(x_{n}, y_{n}) - d(x,y)] \in intP$ for all $n \geq k$. Hence $4(c +r^{'})-[2(c +r^{'}) +d(x,y)- d(x_{n}, y_{n})] \in intP$ for all $n \geq k$, that is $[2(c +r^{'}) +d(x,y)- d(x_{n}, y_{n})]<< 4(c +r^{'})$ for all $n \geq k$. Now by $(iv)$ we can write $0<< 2(c +r^{'})+[d(x,y)- d(x_{n}, y_{n})]$ for all $n \geq k$. Also we have $0<< 4(c +r^{'})$. Hence by normality of $P$ we can write $\left\|  2(c +r^{'})+[d(x,y)- d(x_{n}, y_{n})] \right\| \leq k \left\|4(c +r^{'}) \right\|$ for all $n \geq k \longrightarrow (v)$. Now $\left\| d(x,y)- d(x_{n}, y_{n}) \right\| = \left\|d(x,y)- d(x_{n}, y_{n}) +2(c +r^{'})- 2(c +r^{'}) \right\| \leq \left\|d(x,y)- d(x_{n}, y_{n}) +2(c +r^{'})\right\| + \left\|2(c +r^{'}) \right\|$ for all $n \in \mathbb{N} \longrightarrow (vi)$. From $(v)$ and $(vi)$ we have $\left\| d(x,y)- d(x_{n}, y_{n}) \right\| \leq 4k\left\|(c +r^{'})\right\| + 2\left\|(c +r^{'}) \right\|$ for all $n \geq k$. Also $4k\left\|(c +r^{'})\right\| + 2\left\|(c +r^{'}) \right\| = (4k+2)\left\|(c +r^{'}) \right\| \leq (4k+2) \left\|c \right\| + (4k+2)\left\| r^{'} \right\| < \epsilon + \left\| r \right\|$. Hence $\left\| d(x_{n}, y_{n})-  d(x,y)\right\| \newline < \epsilon + \left\| r \right\|$ for all $n \geq k$.
\end{proof}

\end{document}